\documentclass[12pt,reqno,a4paper]{amsart}%
\usepackage{amsfonts}
\usepackage{amsmath}
\usepackage{amssymb}
\usepackage[utf8]{inputenc}
\usepackage[T1]{fontenc}
\usepackage{lmodern}
\usepackage{graphicx}%
\newtheorem{theorem}{Theorem}[section]
\newtheorem{corollary}[theorem]{Corollary}
\newtheorem{lemma}[theorem]{Lemma}
\theoremstyle{definition}
\newtheorem{definition}[theorem]{Definition}
\newtheorem{remark}[theorem]{Remark}
\numberwithin{equation}{section} 
\begin{document}
\title[On generalized $p-$valent non-Bazilevi$\check{C}$ functions.]{On generalized $p-$valent non-Bazilevi$\mathbf{\check{C}}$ functions of order
$\alpha+i\beta$}
\author{A.A. AMOURAH}
\address{A.A. AMOURAH: Department of Mathematics, Faculty of Science and Technology,
Irbid National University, Irbid, Jordan.}
\email{alaammour@yahoo.com.}
\author{A. G. AlAmoush}
\address{A. G. AlAmoush: School of Mathematical Sciences Faculty of Science and
Technology Universiti Kebangsaan Malaysia Bangi 43600 Selangor D. Ehsan,
Malaysia. }
\email{adnan-omoush@yahoo.com.}
\author{M. DARUS }
\address{M. DARUS : School of Mathematical Sciences Faculty of Science and Technology
Universiti Kebangsaan Malaysia Bangi 43600 Selangor D. Ehsan, Malaysia. }
\email{maslina@ukm.edu.my.}

\begin{abstract}
In this paper, we introduce a subclass $N_{p,\mu}^{n}(\alpha,\beta,A,B)$ of
$p$-valent non-Bazilevi$\check{c}$ functions of order $\alpha+i\beta.$ Some
subordination relations and the inequality properties of $p-$valent functions
are discussed. The results presented here generalize and improve some known results.

\end{abstract}
\maketitle

\section{Introduction and preliminaries}

Let $\mathcal{A}_{p}$ denote the class of functions $f$ of the form:%
\begin{equation}
f(z)=z^{p}+\sum\limits_{k=n}^{\infty}a_{k+p}z^{k+p}\ \ \text{ }(p,n\in%
%TCIMACRO{\U{2115} }%
%BeginExpansion
\mathbb{N}
%EndExpansion
=\left\{  1,2,3,...\right\}  ),\text{ \ \ \ \ } \label{1}%
\end{equation}
which are \emph{analytic} and \emph{$p-$valent} in the open unit disc
$\mathbb{U}=\{z\in\mathbb{C}:$ $\left\vert z\right\vert <1\}.$ If $f(z)$ and
$g(z)$ are analytic in $\mathbb{U},$ we say that $f(z)$ is \emph{subordinate}
to $g(z)$, and we write:
\begin{equation}
f\prec g\text{ in }\mathbb{U}\text{ or }f(z)\prec g(z),\text{ }z\in
\mathbb{U},\bigskip\label{10}%
\end{equation}
if there exists a Schwarz function $w(z)$, which is analytic in $\mathbb{U}$
with \bigskip%
\[
\left\vert w(0)\right\vert =0\text{ and }\left\vert w(z)\right\vert <1,\text{
}z\in\mathbb{U},
\]
such that%
\[
f(z)=g(w(z)),\text{ }z\in\mathbb{U}.
\]

Furthermore, if the function $g(z)$ is \emph{univalent} in $\mathbb{U}$, then
we have the following equivalence, see Miller \& Mocanu (\cite{mill},
\cite{mill1}),%
\[
\text{ }f(z)\prec g(z)\text{ }(z\in\mathbb{U)\Leftrightarrow}f(0)=g(0)\text{
and }f(\mathbb{U)\subset}g(\mathbb{U}).
\]

We define a subclass of $\mathcal{A}_{p}$ as follows:

\begin{definition}
Let $N_{p,\mu}^{n}(\alpha,\beta,A,B)$ denote the class of functions
$f(z)\in\mathcal{A}_{p}$ satisfying the inequality:%
\begin{equation}
\left\{  (1+\mu)\left(  \frac{z^{p}}{f(z)}\right)  ^{\alpha+i\beta}-\mu\left(
\frac{zf^{\prime}(z)}{pf(z)}\right)  \left(  \frac{z^{p}}{f(z)}\right)
^{\alpha+i\beta}\right\}  \prec\frac{1+Az}{1+Bz},\text{ }(z\in\mathbb{U)},
\label{87}%
\end{equation}
where $\mu\in\mathbb{C},$ $\alpha\geq0,$ $\beta\in\mathbb{R},$ $-1\leq
B\leq1,$ $A\neq B,$ $p\in%
%TCIMACRO{\U{2115} }%
%BeginExpansion
\mathbb{N}
%EndExpansion
$ and $A\in%
%TCIMACRO{\U{211d} }%
%BeginExpansion
\mathbb{R}
%EndExpansion
$. All the powers in (\ref{87}) are principal values.
\end{definition}

We say that the function $f(z)$ in this class is $p$-valent
non-Bazilevi$\check{c}$ functions of type $\alpha+i\beta$.

\begin{definition}
Let $f\in$ $N_{p,\mu}^{n}(\alpha,\beta,\rho)$ if and only if $f(z)\in
\mathcal{A}_{p}$ and it satisfies%
\begin{equation}
\operatorname{Re}\left\{  (1+\mu)\left(  \frac{z^{p}}{f(z)}\right)
^{\alpha+i\beta}-\mu\left(  \frac{zf^{\prime}(z)}{pf(z)}\right)  \left(
\frac{z^{p}}{f(z)}\right)  ^{\alpha+i\beta}\right\}  >\rho,\text{ }%
(z\in\mathbb{U)}, \label{88}%
\end{equation}
where $\mu\in\mathbb{C},$ $\alpha\geq0,$ $\beta\in\mathbb{R},$ $p\in%
%TCIMACRO{\U{2115} }%
%BeginExpansion
\mathbb{N}
%EndExpansion
$ and $0\leq\rho<p$.
\end{definition}

\textbf{Special Cases}:

\begin{enumerate}
\item {When $p=1$, then }$N_{1,\mu}^{n}(\alpha,\beta,A,B)${ is the class
studied by AlAmoush and Darus \cite{adnan}.} \newline

\item {When $p=1$, $\beta=0$, then }$N_{1,\mu}^{n}(\alpha,0,A,B)${ is the
class studied by Wang et al \cite{wang}. \newline}

\item {When $p=1$, $\beta=0$, $\mu=-1$, $A=1$ and $B=-1$, then }$N_{\mu}%
^{n}(\alpha)${ is the class studied by Obradovic \cite{Obradovic}. \newline}

\item {If $p=1$, $\beta=0,$ }$\mu${$=B=-1$ and $A=1-2\rho$, then the class
}$N_{1,-1}^{n}(\alpha,0,{1-2\rho},-1)${ reduces to the class of
non-Bazilevi$\check{c}$ functions of order $\rho\ (0\leq\rho<1)$. The
Fekete-Szeg\"{o} problem of the class }$N_{1,-1}^{n}(\alpha,0,{1-2\rho},-1)${
were considered by Tuneski and Darus \cite{maslina}.}
\end{enumerate}

We will need the following lemmas in the next section.

\begin{lemma}
\label{lem1}\cite{miller} Let the function $h(z)$ be analytic and convex in
$\mathbb{U}$ with $h(0)=1$. Suppose also that the function $\Phi(z)$ given by%
\[
\Phi(z)=1+c_{n}z^{n}+c_{n+1}z^{n+1}+...
\]
is analytic in $\mathbb{U}$. If%
\begin{equation}
\Phi(z)+\frac{1}{\gamma}z\Phi^{^{\prime}}(z)\prec h(z)\text{ }(z\in
\mathbb{U},\operatorname{Re}\gamma\geq0,\gamma\neq0), \label{p}%
\end{equation}
then%
\[
\Phi(z)\prec\Psi(z)=\frac{\gamma}{n}z^{-\gamma/n}\int_{0}^{z}t^{(\gamma
/n)-1}h(t)dt\prec h(z)\text{ }(z\in\mathbb{U)},
\]
and $\Psi(z)$ is the best dominant for the differential subordination (\ref{p}).
\end{lemma}

\begin{lemma}
\label{lem2}\cite{qaz} Let $-1\leq B_{1}\leq B_{2}<A_{2}<A_{1}\leq1,$ then%
\[
\frac{1+A_{2}z}{1+B_{2}z}\prec\frac{1+A_{1}z}{1+B_{1}z}.
\]

\end{lemma}

\begin{lemma}
\label{lem3}\cite{liu} Let $\Phi(z)$ be analytic and convex in $\mathbb{U},$
$f(z)\in\mathcal{A}_{p},$ $g(z)\in\mathcal{A}_{p}.$

If $f(z)\prec\Phi(z),$ $g(z)\prec\Phi(z),$ $0\leq\mu\leq1$ then
\[
\mu f(z)+(1-\mu)g(z)\prec\Phi(z).
\]

\end{lemma}

\begin{lemma}
\label{ffder4}\emph{\cite{wx}} Let $q(z)$ be a convex univalent function in
$\mathbb{U}$ and let $\sigma\in\mathbb{C}$, $\eta\in\mathbb{C}^{\ast
}=\mathbb{C}\backslash\left\{  0\right\}  $ with%
\[
\operatorname{Re}\left\{  1+\frac{zq^{^{\prime\prime}}(z)}{q^{^{\prime}}%
(z)}\right\}  >\max\left\{  0,-\operatorname{Re}\left(  \frac{\sigma}{\eta
}\right)  \right\}  .
\]

If the function $\Phi(z)$ is analytic in $\mathbb{U}$ and%
\[
\sigma\Phi(z)+\eta z\Phi^{^{\prime}}(z)\prec\sigma q(z)+\eta zq^{^{\prime}%
}(z),
\]

then, $\Phi(z)\prec q(z)$ and $q(z)$ is the best dominant.
\end{lemma}

\begin{lemma}
\label{ffdernn54}\emph{\cite{Miller3431}} Let $q(z)$ be a convex univalent in
$\mathbb{U}$ and $\eta\in\mathbb{C}$. Further, assume that $\operatorname{Re}%
\left\{  \overline{\eta}\right\}  >0.$ If $\Phi(z)\in H\left[  q(0),1\right]
\cap Q,$ and $\Phi(z)+\eta z\Phi^{^{\prime}}(z)$ is univalent in $\mathbb{U}$.
Then
\[
q(z)+\eta zq^{^{\prime}}(z)\prec\Phi(z)+\eta z\Phi^{^{\prime}}(z),
\]

signifies that $q(z)\prec\Phi(z)$ and $q(z)$ are the best subordinat.
\end{lemma}

We employ techniques similar to these used earlier by Yousef et al. {\cite{1},
}Amourah et al. ({\cite{2}, \cite{3}), AlAmoush and Darus \cite{adnan1} and }Al-Hawary et al. {\cite{1}.}

In the present paper, we shall obtain results concerning the subordination
relations and inequality properties of the class $N_{p,\mu}^{n}(\alpha
,\beta,A,B)$. The results obtained generalize the related works of some authors.

\section{Main Result}

\begin{theorem}
\label{th1444} Let $\mu\in\mathbb{C},$ $\alpha\geq0,$ $\beta\in\mathbb{R},$
$-1\leq B\leq1,$ $A\neq B,$ $\alpha+i\beta\neq0,$ $p\in%
%TCIMACRO{\U{2115} }%
%BeginExpansion
\mathbb{N}
%EndExpansion
$ and $A\in%
%TCIMACRO{\U{211d} }%
%BeginExpansion
\mathbb{R}
%EndExpansion
$. If $f(z)\in N_{p,\mu}^{n}(\alpha,\beta,A,B)$, Then%
\begin{equation}
\left(  \frac{z^{p}}{f(z)}\right)  ^{\alpha+i\beta}\prec\frac{p(\alpha
+i\beta)}{\mu n}\int_{0}^{1}\frac{1+Azu}{1+Bzu}u^{\frac{p(\alpha+i\beta)}{\mu
n}-1}du\prec\frac{1+Az}{1+Bz}.
\end{equation}

\end{theorem}

\begin{proof}
Let%
\begin{equation}
\Phi(z)=\left(  \frac{z^{p}}{f(z)}\right)  ^{\alpha+i\beta}. \label{13}%
\end{equation}

Then $\Phi(z)$ is analytic in $\mathbb{U}$ with $\Phi(0)=1$. Taking
logarithmic differentiation of (\ref{13}) in both sides, we obtain%
\[
p(\alpha+i\beta)-(\alpha+i\beta)\frac{zf(z)^{^{\prime}}}{f(z)}=\frac
{z\Phi^{^{\prime}}(z)}{\Phi(z)}.
\]

In the above equation, we have%
\[
1-\frac{zf^{\prime}(z)}{pf(z)}=\frac{1}{p(\alpha+i\beta)}\frac{z\Phi
^{^{\prime}}(z)}{\Phi(z)}.
\]

From this we can easily deduce that%
\begin{equation}
\left\{  (1+\mu)\left(  \frac{z^{p}}{f(z)}\right)  ^{\alpha+i\beta}-\mu\left(
\frac{zf^{\prime}(z)}{pf(z)}\right)  \left(  \frac{z^{p}}{f(z)}\right)
^{\alpha+i\beta}\right\}  . \label{18}%
\end{equation}

On a class of $p-$valent non-Bazilevi$\check{c}$ functions%
\begin{equation}
=\Phi(z)+\frac{\mu z\Phi^{^{\prime}}(z)}{p(\alpha+i\beta)}\prec\frac
{1+Az}{1+Bz}. \label{w}%
\end{equation}

Now, by Lemma \ref{lem1} for $\gamma=\frac{p(\alpha+i\beta)}{\mu}$ , we deduce
that%
\[
\left(  \frac{z^{p}}{f(z)}\right)  ^{\alpha+i\beta}\prec q(z)=\frac
{p(\alpha+i\beta)}{\mu n}z^{-\frac{p(\alpha+i\beta)}{\mu n}}\int_{0}%
^{z}t^{\frac{p(\alpha+i\beta)}{\mu n}-1}(\frac{1+At}{1+Bt})dt.
\]

Putting $t=zu\Rightarrow dt=zdu.$ Then we have the above equation with%
\[
=\frac{p(\alpha+i\beta)}{\mu n}\int_{0}^{1}\frac{1+Azu}{1+Bzu}u^{\frac
{p(\alpha+i\beta)}{\mu n}-1}du\prec\frac{1+Az}{1+Bz},
\]
and the proof is complete.
\end{proof}

\begin{corollary}
\label{corl12} Let $\mu\in\mathbb{C},$ $\alpha\geq0,$ $\beta\in\mathbb{R},$
$\alpha+i\beta\neq0,$ $p\in%
%TCIMACRO{\U{2115} }%
%BeginExpansion
\mathbb{N}
%EndExpansion
$ and $\rho\neq0.$ If $f(z)\in\mathcal{A}_{p}$ satisfies%
\[
(1+\mu)\left(  \frac{z^{p}}{f(z)}\right)  ^{\alpha+i\beta}-\mu\left(
\frac{zf^{\prime}(z)}{pf(z)}\right)  \left(  \frac{z^{p}}{f(z)}\right)
^{\alpha+i\beta}\prec\frac{1+(1-2\rho)z}{1-z},\text{ }(z\in\mathbb{U)},
\]

then%
\[
\left(  \frac{z^{p}}{f(z)}\right)  ^{\alpha+i\beta}\prec\frac{p(\alpha
+i\beta)}{\mu n}\int_{0}^{1}\frac{1+(1-2\rho)zu}{1-zu}u^{\frac{p(\alpha
+i\beta)}{\mu n}-1}du,\text{ }(z\in\mathbb{U)},
\]

or equivalent to%
\[
\left(  \frac{z^{p}}{f(z)}\right)  ^{\alpha+i\beta}\prec\rho+\frac
{p(\alpha+i\beta)(1-\rho)}{\mu n}\int_{0}^{1}\frac{1+zu}{1-zu}u^{\frac
{p(\alpha+i\beta)}{\mu n}-1}du,\text{ }(z\in\mathbb{U)}.
\]

\end{corollary}

\begin{corollary}
\label{corl1} Let $\mu\in\mathbb{C},$ $\alpha\geq0,$ $\beta\in\mathbb{R},$
$\alpha+i\beta\neq0,$ $p\in%
%TCIMACRO{\U{2115} }%
%BeginExpansion
\mathbb{N}
%EndExpansion
$ and $\operatorname{Re}\left\{  \mu\right\}  \geq0,$ then%
\[
N_{p,\mu}^{n}(\alpha,\beta,A,B)\subset N_{p,0}^{n}(\alpha,\beta,A,B).
\]

\end{corollary}

\begin{theorem}
\label{thm221} Let $0\leq\mu_{1}\leq\mu_{2},$ $\alpha\geq0,$ $\beta
\in\mathbb{R},$ $p\in%
%TCIMACRO{\U{2115} }%
%BeginExpansion
\mathbb{N}
%EndExpansion
,$ $\alpha+i\beta\neq0\ $and $-1\leq B_{1}\leq B_{2}<A_{2}\leq A_{1}\leq1,$
then%
\begin{equation}
N_{p,\mu_{2}}^{n}(\alpha,\beta,A_{2},B_{2})\subset N_{p,\mu_{1}}^{n}%
(\alpha,\beta,A_{1},B_{1}). \label{rr}%
\end{equation}

\end{theorem}

\begin{proof}
Suppose that $f(z)\in N_{p,\mu_{2}}^{n}(\alpha,\beta,A_{2},B_{2})$ we have
$f(z)\in\mathcal{A}_{p}$ and%
\[
\left\{  (1+\mu_{2})\left(  \frac{z^{p}}{f(z)}\right)  ^{\alpha+i\beta}%
-\mu_{2}\left(  \frac{zf^{\prime}(z)}{pf(z)}\right)  \left(  \frac{z^{p}%
}{f(z)}\right)  ^{\alpha+i\beta}\right\}  \prec\frac{1+A_{2}z}{1+B_{2}%
z},\text{ }(z\in\mathbb{U)}.
\]

Since $-1\leq B_{1}\leq B_{2}<A_{2}\leq A_{1}\leq1,$ therefore it follows from
Lemma \ref{lem2} that%
\begin{equation}
\left\{  (1+\mu_{2})\left(  \frac{z^{p}}{f(z)}\right)  ^{\alpha+i\beta}%
-\mu_{2}\left(  \frac{zf^{\prime}(z)}{pf(z)}\right)  \left(  \frac{z^{p}%
}{f(z)}\right)  ^{\alpha+i\beta}\right\}  \prec\frac{1+A_{1}z}{1+B_{1}%
z},\text{ }(z\in\mathbb{U)}, \label{26}%
\end{equation}

that is $f(z)\in N_{p,\mu_{2}}^{n}(\alpha,\beta,A_{1},B_{1}).$ So Theorem
\ref{thm221} is proved when $\mu_{1}=\mu_{2}\geq0.$

When $\mu_{2}>\mu_{1}\geq0,$ then we can see from Corollary \ref{corl1} that
$f(z)\in N_{p,0}^{n}(\alpha,\beta,A_{1},B_{1}),$ then%
\begin{equation}
\left(  \frac{z^{p}}{f(z)}\right)  ^{\alpha+i\beta}\prec\frac{1+A_{1}%
z}{1+B_{1}z}. \label{27}%
\end{equation}

But%
\begin{align*}
&  \left\{  (1+\mu_{1})\left(  \frac{z^{p}}{f(z)}\right)  ^{\alpha+i\beta}%
-\mu_{1}\left(  \frac{zf^{\prime}(z)}{pf(z)}\right)  \left(  \frac{z^{p}%
}{f(z)}\right)  ^{\alpha+i\beta}\right\} \\
&  =(1-\frac{\mu_{1}}{\mu_{2}})\left(  \frac{z^{p}}{f(z)}\right)
^{\alpha+i\beta}+\frac{\mu_{1}}{\mu_{2}}\left\{
\begin{array}
[c]{c}%
(1+\mu_{1})\left(  \frac{z^{p}}{f(z)}\right)  ^{\alpha+i\beta}\\
-\mu_{1}\left(  \frac{zf^{\prime}(z)}{pf(z)}\right)  \left(  \frac{z^{p}%
}{f(z)}\right)  ^{\alpha+i\beta}%
\end{array}
\right\}  .
\end{align*}

It is obvious that $\frac{1+A_{1}z}{1+B_{1}z}$ is analytic and convex in
$\mathbb{U}.$ Sowe obtain fromLemma \ref{lem3} and differential subordinations
(\ref{26}) and (\ref{27}) that%
\[
\left\{  (1+\mu_{1})\left(  \frac{z^{p}}{f(z)}\right)  ^{\alpha+i\beta}%
-\mu_{1}\left(  \frac{zf^{\prime}(z)}{pf(z)}\right)  \left(  \frac{z^{p}%
}{f(z)}\right)  ^{\alpha+i\beta}\right\}  \frac{1+A_{1}z}{1+B_{1}z},
\]

that is, $f(z)\in N_{p,\mu_{1}}^{n}(\alpha,\beta,A_{1},B_{1}).$ Thuswe have%
\[
N_{p,\mu_{2}}^{n}(\alpha,\beta,A_{2},B_{2})\subset N_{p,\mu_{1}}^{n}%
(\alpha,\beta,A_{1},B_{1}).
\]
\ \ \ \ \ \ \ \ \ \ \ \ \ \ \ \ \ \ \ \ \ \ \ \ \ \ \ \ \ \ \ \ \ \ \ \ \ \ \ \ \ \ \ \ \ \ \ \ \ \ \ \ \ \ \ \ \ \ \ \ \ \ \ \ \ \ \ \ \ \ \ \ \ \ \ \ \ \ \ \ \ \ \ \ \ \ \ \ \ \ \ \ \ \ \ \ \
\end{proof}

\begin{corollary}
\label{cor13} Let $0\leq\mu_{1}\leq\mu_{2},$ $0\leq\rho_{1}\leq\rho_{2},$
$\alpha\geq0,$ $\beta\in\mathbb{R},$ $p\in%
%TCIMACRO{\U{2115} }%
%BeginExpansion
\mathbb{N}
%EndExpansion
$ and $\alpha+i\beta\neq0$ then%
\[
N_{p,\mu_{2}}^{n}(\alpha,\beta,\rho_{2})\subset N_{p,\mu_{1}}^{n}(\alpha
,\beta,\rho_{1}).
\]

\end{corollary}

\begin{theorem}
\bigskip\label{thm222} Let $\mu\in%
%TCIMACRO{\U{2102} }%
%BeginExpansion
\mathbb{C}
%EndExpansion
,$ $\alpha\geq0,$ $\beta\in%
%TCIMACRO{\U{211d} }%
%BeginExpansion
\mathbb{R}
%EndExpansion
$, $\mu+i\beta\neq0,$ $p\in%
%TCIMACRO{\U{2115} }%
%BeginExpansion
\mathbb{N}
%EndExpansion
,$ $-1\leq B\leq1,$ $A\neq B$ and $A\in%
%TCIMACRO{\U{211d} }%
%BeginExpansion
\mathbb{R}
%EndExpansion
.$ If $f(z)\in N_{p,\mu}^{n}(\alpha,\beta,A,B),$ then%
\begin{align*}
&  \inf_{z\in\mathbb{U}}\operatorname{Re}\left\{  \frac{p(\alpha+i\beta)}{\mu
n}\int_{0}^{1}\frac{1+Azu}{1+Bzu}u^{\frac{p(\alpha+i\beta)}{\mu n}%
-1}du\right\} \\
&  <\operatorname{Re}\left(  \frac{z^{p}}{f(z)}\right)  ^{\alpha+i\beta}%
<\sup_{z\in\mathbb{U}}\operatorname{Re}\left\{  \frac{p(\alpha+i\beta)}{\mu
n}\int_{0}^{1}\frac{1+Azu}{1+Bzu}u^{\frac{p(\alpha+i\beta)}{\mu n}%
-1}du\right\}  .
\end{align*}

\end{theorem}

\begin{proof}
Suppose that $f(z)\in N_{p,\mu}^{n}(\alpha,\beta,A,B),$ then from Theorem
\ref{th1444} we know that%
\begin{equation}
\left(  \frac{z^{p}}{f(z)}\right)  ^{\alpha+i\beta}\prec\frac{p(\alpha
+i\beta)}{\mu n}\int_{0}^{1}\frac{1+Azu}{1+Bzu}u^{\frac{p(\alpha+i\beta)}{\mu
n}-1}du. \label{223}%
\end{equation}

Therefore, from the definition of the subordination, we have%
\[
\operatorname{Re}\left(  \frac{z^{p}}{f(z)}\right)  ^{\alpha+i\beta}%
>\inf_{z\in\mathbb{U}}\operatorname{Re}\left\{  \frac{p(\alpha+i\beta)}{\mu
n}\int_{0}^{1}\frac{1+Azu}{1+Bzu}u^{\frac{p(\alpha+i\beta)}{\mu n}%
-1}du\right\}  ,
\]%
\[
\operatorname{Re}\left(  \frac{z^{p}}{f(z)}\right)  ^{\alpha+i\beta}%
<\sup_{z\in\mathbb{U}}\operatorname{Re}\left\{  \frac{p(\alpha+i\beta)}{\mu
n}\int_{0}^{1}\frac{1+Azu}{1+Bzu}u^{\frac{p(\alpha+i\beta)}{\mu n}%
-1}du\right\}  .
\]

\end{proof}

\begin{corollary}
\label{cor1} Let $\mu\in%
%TCIMACRO{\U{2102} }%
%BeginExpansion
\mathbb{C}
%EndExpansion
,$ $\alpha\geq0,$ $\beta\in%
%TCIMACRO{\U{211d} }%
%BeginExpansion
\mathbb{R}
%EndExpansion
$, $\mu+i\beta\neq0,$ $p\in%
%TCIMACRO{\U{2115} }%
%BeginExpansion
\mathbb{N}
%EndExpansion
$ and $\rho<1.$ If $f(z)\in N_{p,\mu}^{n}(\alpha,\beta,1-2\rho,-1),$ then%
\begin{align}
&  \rho+\left(  1-\rho\right)  \inf_{z\in\mathbb{U}}\operatorname{Re}\left\{
\frac{p(\alpha+i\beta)}{\mu n}\int_{0}^{1}\frac{1+zu}{1-zu}u^{\frac
{p(\alpha+i\beta)}{\mu n}-1}du\right\} \nonumber\\
&  <\operatorname{Re}\left(  \frac{z^{p}}{f(z)}\right)  ^{\alpha+i\beta
}\nonumber\\
&  <\rho+\left(  1-\rho\right)  \sup_{z\in\mathbb{U}}\operatorname{Re}\left\{
\frac{p(\alpha+i\beta)}{\mu n}\int_{0}^{1}\frac{1+zu}{1-zu}u^{\frac
{p(\alpha+i\beta)}{\mu n}-1}du\right\}  . \label{35}%
\end{align}

\end{corollary}

\begin{corollary}
\label{cor2} Let $\mu\in%
%TCIMACRO{\U{2102} }%
%BeginExpansion
\mathbb{C}
%EndExpansion
,$ $\alpha\geq0,$ $\beta\in%
%TCIMACRO{\U{211d} }%
%BeginExpansion
\mathbb{R}
%EndExpansion
$, $\mu+i\beta\neq0,$ $p\in%
%TCIMACRO{\U{2115} }%
%BeginExpansion
\mathbb{N}
%EndExpansion
$ and $\rho>1.$ If $f(z)\in N_{p,\mu}^{n}(\alpha,\beta,1-2\rho,-1),$ then%
\[
\operatorname{Re}\left\{  (1+\mu_{1})\left(  \frac{z^{p}}{f(z)}\right)
^{\alpha+i\beta}-\mu_{1}\left(  \frac{zf^{\prime}(z)}{pf(z)}\right)  \left(
\frac{z^{p}}{f(z)}\right)  ^{\alpha+i\beta}\right\}  <\rho,\text{ }%
z\in\mathbb{U},
\]
then%
\begin{align*}
&  \rho+\left(  1-\rho\right)  \sup_{z\in\mathbb{U}}\operatorname{Re}\left\{
\frac{p(\alpha+i\beta)}{\mu n}\int_{0}^{1}\frac{1+zu}{1-zu}u^{\frac
{p(\alpha+i\beta)}{\mu n}-1}du\right\} \\
&  <\operatorname{Re}\left(  \frac{z^{p}}{f(z)}\right)  ^{\alpha+i\beta}\\
&  <\rho+\left(  1-\rho\right)  \inf_{z\in\mathbb{U}}\operatorname{Re}\left\{
\frac{p(\alpha+i\beta)}{\mu n}\int_{0}^{1}\frac{1+zu}{1-zu}u^{\frac
{p(\alpha+i\beta)}{\mu n}-1}du\right\}  .
\end{align*}

\end{corollary}

Next, several new differential suborination results of non-Bazilevi$\check{c}$
class of order $\alpha+i\beta$ are established.

\section{Suborination for $p-$Valent Non-Bazilevi$\check{c}$ Class of Order
$\alpha+i\beta$}

By employing lemma \ref{ffder4}, the following result is given.

\begin{theorem}
\label{nana} Let $q$ be univalent in $\mathbb{U}$, $\mu\in\mathbb{C}^{\ast},$
$\alpha\geq0,$ $\beta\in\mathbb{R}$ and $\alpha+i\beta\neq0$. Suppose that $q$
satisfies%
\begin{equation}
\operatorname{Re}\left\{  1+\frac{zq^{^{\prime\prime}}(z)}{q^{^{\prime}}%
(z)}\right\}  >\max\left\{  0,-\operatorname{Re}\left\{  \frac{p(\alpha
+i\beta)}{\mu}\right\}  \right\}  . \label{kkkm}%
\end{equation}

If $f\in\mathcal{A}_{p}$ satisfies the subordination%
\begin{equation}
\left\{  (1+\mu)\left(  \frac{z^{p}}{f(z)}\right)  ^{\alpha+i\beta}-\mu\left(
\frac{zf^{\prime}(z)}{pf(z)}\right)  \left(  \frac{z^{p}}{f(z)}\right)
^{\alpha+i\beta}\right\}  \prec q(z)+\mu\frac{zq^{^{\prime}}(z)}%
{p(\alpha+i\beta)}, \label{kkmkl}%
\end{equation}

then
\begin{equation}
\left(  \frac{z^{p}}{f(z)}\right)  ^{\alpha+i\beta}\prec q(z) \label{mmjn}%
\end{equation}

and $q(z)$ is the best dominant.
\end{theorem}

\begin{proof}
Define the function $\Phi(z)$ by%
\begin{equation}
\Phi(z)=\left(  \frac{z^{p}}{f(z)}\right)  ^{\alpha+i\beta}. \label{112nh}%
\end{equation}

Having (\ref{112nh}) differentiated logarithmically in connection with $z$, we
have%
\begin{equation}
\frac{z\Phi^{^{\prime}}(z)}{\Phi(z)}=p(\alpha+i\beta)\left(  1-\frac
{zf(z)^{^{\prime}}}{pf(z)}\right)  , \label{bby6}%
\end{equation}

which, with respect to hypothesis (\ref{kkmkl}) of Theorem \ref{nana}, the
following subordination is obtained:%
\begin{equation}
\Phi(z)+\frac{\mu z\Phi^{^{\prime}}(z)}{p(\alpha+i\beta)}\prec q(z)+\mu
\frac{zq^{^{\prime}}(z)}{p(\alpha+i\beta)}. \label{mjhnb}%
\end{equation}

Theorem \ref{nana} assertion is now followed by the use of Lemma \ref{ffder4}
with $\eta=\frac{\mu}{p(\alpha+i\beta)}$ and $\sigma=1.$
\end{proof}

\begin{remark}
\label{remtt} For the choice $q(z)=\frac{1+Az}{1+Bz}$ in Theorem \ref{nana},
the following the corollary is obtained.
\end{remark}

\begin{corollary}
\label{coryy} Let $\mu\in\mathbb{C}^{\ast},$ $-1\leq B<A\leq1,$ and%
\begin{equation}
\operatorname{Re}\left\{  \frac{1-Bz}{1+Bz}\right\}  >\max\left\{
0,-\operatorname{Re}\left\{  \frac{p(\alpha+i\beta)}{\mu}\right\}  \right\}
,\text{ }(z\in\mathbb{U}). \label{webh5}%
\end{equation}

If $f\in\mathcal{A}_{p},$ and%
\begin{equation}
\left\{  (1+\mu)\left(  \frac{z^{p}}{f(z)}\right)  ^{\alpha+i\beta}-\mu\left(
\frac{zf^{\prime}(z)}{pf(z)}\right)  \left(  \frac{z^{p}}{f(z)}\right)
^{\alpha+i\beta}\right\}  \prec\frac{\mu(A-B)z}{p(\alpha+i\beta)(1+Bz)^{2}%
}+\frac{1+Az}{1+Bz}, \label{ddffd}%
\end{equation}

then%
\begin{equation}
\left(  \frac{z^{p}}{f(z)}\right)  ^{\alpha+i\beta}\prec\frac{1+Az}{1+Bz}
\label{eew3}%
\end{equation}

and $\frac{1+Az}{1+Bz}$ is the best dominant.
\end{corollary}

\begin{remark}
\label{remvv} For the choice $q(z)=\frac{1+z}{1-z}$ in Theorem \ref{nana}, the
following the corollary is obtained.
\end{remark}

\begin{corollary}
\label{cornnb} Let $\mu\in\mathbb{C}^{\ast},$ and
\begin{equation}
\operatorname{Re}\left\{  \frac{1+z}{1-z}\right\}  >\max\left\{
0,-\operatorname{Re}\left\{  \frac{p(\alpha+i\beta)}{\mu}\right\}  \right\}
,\text{ }(z\in\mathbb{U}). \label{224aas}%
\end{equation}

If $f\in\mathcal{A}_{p},$ and%
\begin{equation}
\left\{  (1+\mu)\left(  \frac{z^{p}}{f(z)}\right)  ^{\alpha+i\beta}-\mu\left(
\frac{zf^{\prime}(z)}{pf(z)}\right)  \left(  \frac{z^{p}}{f(z)}\right)
^{\alpha+i\beta}\right\}  \prec\frac{2\mu z}{p(\alpha+i\beta)(1-z)^{2}}%
+\frac{1+z}{1-z}, \label{ssw3}%
\end{equation}

then%
\begin{equation}
\left(  \frac{z^{p}}{f(z)}\right)  ^{\alpha+i\beta}\prec\frac{1+z}{1-z}
\label{weaq123}%
\end{equation}

and $\frac{1+Az}{1+Bz}$ is the best dominant.
\end{corollary}

\section{Superordination for $p-$Valent Non-Bazilevi$\check{c}$ Class of Order
$\alpha+i\beta$}

\begin{theorem}
\label{nanao} Let $q$ be convex univalent in $\mathbb{U}$, $\mu\in\mathbb{C},$
$\alpha\geq0,$ $\beta\in\mathbb{R}$ and $\alpha+i\beta\neq0$. Suppose that $q$
satisfies%
\begin{equation}
\operatorname{Re}\left\{  \mu\right\}  >0 \label{d34}%
\end{equation}

and $\left(  \frac{z^{p}}{f(z)}\right)  ^{\alpha+i\beta}\in H\left[
q(0),1\right]  \cap Q.$ Let%
\begin{equation}
(1+\mu)\left(  \frac{z^{p}}{f(z)}\right)  ^{\alpha+i\beta}-\mu\left(
\frac{zf^{\prime}(z)}{pf(z)}\right)  \left(  \frac{z^{p}}{f(z)}\right)
^{\alpha+i\beta} \label{kk98}%
\end{equation}

be univalent in $\mathbb{U}$, If
\begin{equation}
q(z)+\mu\frac{zq^{^{\prime}}(z)}{p(\alpha+i\beta)}\prec(1+\mu)\left(
\frac{z^{p}}{f(z)}\right)  ^{\alpha+i\beta}-\mu\left(  \frac{zf^{\prime}%
(z)}{pf(z)}\right)  \left(  \frac{z^{p}}{f(z)}\right)  ^{\alpha+i\beta},
\label{445rf}%
\end{equation}

then
\begin{equation}
q(z)\prec\left(  \frac{z^{p}}{f(z)}\right)  ^{\alpha+i\beta} \label{jjhg8}%
\end{equation}

and $q(z)$ is the best subordinant.
\end{theorem}

\begin{proof}
Define the function $\Phi(z)$ by%
\begin{equation}
\Phi(z)=\left(  \frac{z^{p}}{f(z)}\right)  ^{\alpha+i\beta}. \label{rrfg564}%
\end{equation}

Then based on a computation, it indicates that%
\begin{equation}
\Phi(z)+\frac{\mu z\Phi^{^{\prime}}(z)}{p(\alpha+i\beta)}\prec(1+\mu)\left(
\frac{z^{p}}{f(z)}\right)  ^{\alpha+i\beta}-\mu\left(  \frac{zf^{\prime}%
(z)}{pf(z)}\right)  \left(  \frac{z^{p}}{f(z)}\right)  ^{\alpha+i\beta}.
\label{ffdre345}%
\end{equation}

Theorem \ref{nanao} follows from Lemma \ref{ffdernn54}.
\end{proof}

\begin{remark}
\label{remnnh5} Taking $q(z)=\frac{1+Az}{1+Bz}$ in Theorem \ref{nanao}, the
following the corollary is obtained.
\end{remark}

\begin{corollary}
\label{cornnbbb} Let $-1\leq B<A\leq1.$ Let $q$ be convex univalent in
$\mathbb{U}$. Suppose that $q$ satisfies $\operatorname{Re}\left\{
\mu\right\}  >0,$ and $\left(  \frac{z^{p}}{f(z)}\right)  ^{\alpha+i\beta}\in
H\left[  q(0),1\right]  \cap Q.$ Let%
\begin{equation}
(1+\mu)\left(  \frac{z^{p}}{f(z)}\right)  ^{\alpha+i\beta}-\mu\left(
\frac{zf^{\prime}(z)}{pf(z)}\right)  \left(  \frac{z^{p}}{f(z)}\right)
^{\alpha+i\beta} \label{yyu765}%
\end{equation}

be univalent in $\mathbb{U}$, If%
\begin{equation}
\frac{\mu(A-B)z}{p(\alpha+i\beta)(1+Bz)^{2}}+\frac{1+Az}{1+Bz}\prec\left\{
(1+\mu)\left(  \frac{z^{p}}{f(z)}\right)  ^{\alpha+i\beta}-\mu\left(
\frac{zf^{\prime}(z)}{pf(z)}\right)  \left(  \frac{z^{p}}{f(z)}\right)
^{\alpha+i\beta}\right\}  , \label{aa34}%
\end{equation}

then%
\begin{equation}
\frac{1+Az}{1+Bz}\prec\left(  \frac{z^{p}}{f(z)}\right)  ^{\alpha+i\beta}%
\end{equation}

and $\frac{1+Az}{1+Bz}$ is the best subordinant.
\end{corollary}

\section{Sandwich Results for $p-$Valent Non-Bazilevi$\check{c}$ Class of
Order $\alpha+i\beta$}

Combining the differential subordination and supordination results, the
sandwich results are highlighted as follows.

\begin{theorem}
\label{nanaovs} Let $q_{1}$ be convex univalent and let $q_{2}$ be univalent
in $\mathbb{U}$, $\mu\in\mathbb{C},$ $\alpha\geq0,$ $\beta\in\mathbb{R}$ and
$\alpha+i\beta\neq0$. Suppose $q_{1}$ satisfies (\ref{d34}) and $q_{2}$
satisfies (\ref{kkkm}). If $0\neq\left(  \frac{z^{p}}{f(z)}\right)
^{\alpha+i\beta}\in H\left[  q(0),1\right]  \cap Q,$ $(1+\mu)\left(
\frac{z^{p}}{f(z)}\right)  ^{\alpha+i\beta}-\mu\left(  \frac{zf^{\prime}%
(z)}{pf(z)}\right)  \left(  \frac{z^{p}}{f(z)}\right)  ^{\alpha+i\beta}$ is
univalent in $\mathbb{U}$, and if $f\in\mathcal{A}_{p}$ satisfies%
\begin{equation}
q_{1}(z)+\mu\frac{zq_{1}^{^{\prime}}(z)}{p(\alpha+i\beta)}\prec(1+\mu)\left(
\frac{z^{p}}{f(z)}\right)  ^{\alpha+i\beta}-\mu\left(  \frac{zf^{\prime}%
(z)}{pf(z)}\right)  \left(  \frac{z^{p}}{f(z)}\right)  ^{\alpha+i\beta}\prec
q_{2}(z)+\mu\frac{zq_{2}^{^{\prime}}(z)}{p(\alpha+i\beta)}, \label{1124eds}%
\end{equation}

then
\begin{equation}
q_{1}(z)\prec\left(  \frac{z^{p}}{f(z)}\right)  ^{\alpha+i\beta}\prec q_{2}(z)
\label{qqwvvc4}%
\end{equation}

and $q_{1}(z)$ and $q_{2}(z)$ are respectively, the best subordinant and best dominant.
\end{theorem}

\begin{proof}
Simultaneously applying the techniques of the proof of Theorem \ref{nana} and
Theorem \ref{nanao}.
\end{proof}

\textbf{Acknowledgement:} The work here is supported by FRGSTOPDOWN/2013/ST06/UKM/01/1.

\end{document}